\documentclass[10pt,reqno]{amsart}

\usepackage[pdftex]{graphicx} 
\usepackage[english]{babel} 
\usepackage[pdftex,linkcolor=black,pdfborder={0 0 0}]{hyperref} 
\usepackage{calc} 
\usepackage{enumitem} 
\usepackage{comment}

\usepackage{import}

\usepackage{adjustbox}

\usepackage{amsmath}
\usepackage{amssymb}
\usepackage{amsthm}
\usepackage[all]{xy}
\usepackage[pdftex]{graphicx}
\usepackage{color}
\usepackage{cite}
\usepackage{url}
\usepackage{indent first}
\usepackage[labelfont=bf,labelsep=period,justification=raggedright]{caption}
\usepackage[english]{babel}
\usepackage[utf8]{inputenc}
\usepackage[colorinlistoftodos]{todonotes}
\usepackage{tkz-fct}
\usepackage{tikz}
\usetikzlibrary{calc}
\usepackage{comment}

\usepackage{mathtools} 
\mathtoolsset{showonlyrefs} 

\usepackage{multicol}
\PassOptionsToPackage{dvipsnames,svgnames}{xcolor}
\usepackage{textcomp}
\usepackage{sistyle}
\SIthousandsep{,}

\usepackage{lipsum} 
\usepackage{import}

\newcommand{\NN}{\mathbb{N}}

\theoremstyle{plain}
\newtheorem{thm}{Theorem}
\newtheorem{cor}[thm]{Corollary}

\newtheorem{prop}[thm]{Proposition}

\newtheorem{lem}[thm]{Lemma}

\newcommand{\imod}[1]{\,\left(\textnormal{mod }#1\right)\,}

\begin{document}

\title{Additive Representations of Natural Numbers}

\author[F.~J.~Francis]{Forrest J. Francis}
	\address{School of Science, UNSW Canberra, Northcott Drive, Australia ACT 2612} 
	\email{f.francis@student.adfa.edu.au}

\author[E.~S.~Lee]{Ethan S. Lee}
	\address{School of Science, UNSW Canberra, Northcott Drive, Australia ACT 2612} 
	\email{ethan.s.lee@student.adfa.edu.au}

\maketitle

\begin{abstract}
Every natural number greater than two may be written as the sum of a prime and a square-free number. We establish several generalisations of this, by placing divisibility conditions on the square-free number.
\end{abstract}

\section{Introduction}

In 1742, Goldbach conjectured that every even integer greater than two is the sum of two primes.
In 1966, Chen \cite{Chen66,Chen78} proved that every sufficiently large even integer is the sum of one prime and a product of at most two primes\footnote{Yamada \cite{Yamada2015} has shown that Chen's theorem holds for even integers larger than $\exp(\exp(38))$.}, and in 2013, Helfgott \cite{helfgott2013} proved the ternary Goldbach conjecture, which states every odd integer larger than five is the sum of three primes.

Since a complete proof of the Goldbach conjecture remains out of reach, we consider results where we relax one of the primes to be a square-free number instead. To this end, Dudek \cite{dudek2017} proved the following version of Estermann's result \cite{estermann1931}.

\begin{thm}[Dudek, 2017]\label{thm:dudek}
Every integer greater than two is the sum of a prime and a square-free number.
\end{thm}

A simple extension of Theorem \ref{thm:dudek} is given in Corollary \ref{thm:prime_prime_sq_free_theorem}.\footnote{To establish Corollary \ref{thm:prime_prime_sq_free_theorem}, observe that if $n > 4$, then $n-2 >2$, so Theorem \ref{thm:dudek} implies that $n-2 = p + \eta$ for at least one prime $p$ and square-free number $\eta$.}

\begin{cor}\label{thm:prime_prime_sq_free_theorem}
Every integer greater than four may be written as the sum of two primes and a square-free number. 
\end{cor}

As an extension of Theorem \ref{thm:dudek}, Yau \cite{yau2019} established a uniform bound for the number of representations of an integer as a prime in a fixed residue class plus a square-free number.
Instead of placing constraints on the prime, we will impose an additional condition on the divisors of the square-free numbers in Theorem \ref{thm:dudek}.
That is, suppose $q$ is prime and $n > n_0$ where $n_0$ is small. Does there exist at least one prime $p$ and a square-free integer $\eta$ which is co-prime to $q$ and
\begin{equation}\label{eqn:referenceme_coprime_q}\tag{Q1}
    n = p + \eta \, ?
\end{equation}
Moreover, does there exist at least one pair of primes $p_1$, $p_2$ and a square-free integer $\eta$ such that $(\eta, q)=1$ and
\begin{equation}\label{eqn:referenceme_two_coprime_q}\tag{Q2}
    n = p_1 + p_2 + \eta \, ?
\end{equation}


We have answered \eqref{eqn:referenceme_coprime_q} and \eqref{eqn:referenceme_two_coprime_q} for primes $2\leq q < 10^5$ in the following results. Our results are stated for the best possible range of $n$.

\begin{thm}\label{thm:one_prime_two}
Every even integer greater than three can be written as the sum of a prime and an odd square-free number.\footnote{For sufficiently large even integers, Theorem \ref{thm:one_prime_two} is also a consequence of Chen's theorem.}
\end{thm}

\begin{thm}\label{thm:one_prime_three}
Every integer greater than two except for eleven can be written as the sum of a prime and a square-free number which is co-prime to three.
\end{thm}

\begin{thm}\label{thm:one_prime_general}
Suppose $3 < q < 10^5$ is prime. Every integer greater than two can be written as the sum of a prime and a square-free number co-prime to $q$.
\end{thm}

\begin{cor}\label{thm:two_prime_general}
Suppose $2 \leq q < 10^5$ is prime. Every integer greater than four can be written as the sum of two primes and a square-free number co-prime to $q$.
\end{cor}

Computations suggest that one could establish similar results for composite $q$. The authors also believe that a similar method of proof is plausible, if the auxiliary results in Section \ref{ssec:aux_results} can be extended in the appropriate manner.

In future work, it may also be interesting to investigate the quantity $\max{S_q}$, where an \textit{exception set} $S_q$ contains all the integers which do not have a representation as a prime plus a square-free number co-prime to $q$ for any integer $q > 1$.  For example, Theorem \ref{thm:one_prime_three} implies $\max{S_{3}} = 11$, and a search over the first $10^8$ integers suggests that $\max{S_{15}} = 23$, $\max{S_{\prod_{i=2}^{35} p_i}} = 355$. Here, $p_i$ denotes the $i^\text{th}$ prime.

\subsection*{Outline of the paper}

In Section \ref{sec:aux_results}, we provide all of the necessary notation and auxiliary results which will be needed throughout the paper.
In Section \ref{sec:analytic_bounds}, we will prove an important lemma.
Finally, in Section \ref{sec:main_results}, we will prove the main results of this paper.

\subsection*{Acknowledgements}

We would like to thank Tim Trudgian for his comments and bringing this project to our attention. We would also like to thank Nathan Ng and Stephan Garcia for their feedback.

\section{Notation and Auxiliary results}\label{sec:aux_results}

\subsection{Notation}

Throughout, $p$ will denote a prime number and $n$ will denote an integer. Further, $\varphi$ denotes the Euler-phi function, $\mu$ denotes the M\"{o}bius function,
\begin{align*}
    \theta(x)&=\sum_{p\leq x} \log p,\quad
    \theta(x;q,a)=\sum_{\substack{{p\leq x}\\{p \equiv a \imod{q}}}} \log p,\quad 
    \mu_2(n) = \sum_{a^2|n}\mu(a),\\
    R(n) &= \sum_{p \leq n} \mu_2(n-p) \log{p} = \sum_{a \leq n^{\frac{1}{2}}} \mu(a) \theta(n;a^2,n),
\end{align*}
in which $\mu_2(n)=1$ if $n$ is square-free, and $\mu_2(n)=0$ otherwise. We also use $\mu^2$ as a square-free identifier function and $(a,b)$ as the greatest common divisor function.

\subsection{Auxiliary results}\label{ssec:aux_results}

In Section \ref{sec:analytic_bounds}, we will determine estimates for $R(n)$. To do this, we will appeal to the following estimate, which follows from the work of Bennett et al. \cite{bennett2018}.

\begin{prop}\label{prop:bennett_prop}
For each square $a^2\in [2^2, 316^2]$ and integer $n$ which is co-prime to $a$, there exist explicit constants $c_{\theta}(a^2)$ and $x_{\theta}(a^2) \leq 4.81\cdot 10^9$ such that
$$\left| \theta(x;a^2,x) - \frac{x}{\varphi(a^2)} \right| < c_{\theta}(a^2)\frac{x}{\log x}$$
for all $x\geq x_{\theta}(a^2)$.
\end{prop}

\begin{proof}
For each $3\leq q\leq 10^5$ and integers $a$ such that $(a,q) = 1$, Bennett et al. \cite[Theorem 1.2]{bennett2018} provide explicit constants $c_{\theta}(q)$ and $x_{\theta}(q) \leq 8\cdot 10^9$ such that
$$\left| \theta(x;q,a) - \frac{x}{\varphi(q)} \right| < c_{\theta}(q)\frac{x}{\log x}$$
for all $x\geq x_{\theta}(q)$. Analysis on the values of $c_{\theta}(q)$ and $x_{\theta}(q)$ from the tables\footnote{The tables are available at \url{https://www.nt.math.ubc.ca/BeMaObRe/}.} provided for \cite{bennett2018} at each square $q = a^2$ in this range will yield the constants $c_{\theta}(a^2)$ and demonstrate that the maximum value of $x_{\theta}(a^2)$ is $\num{4800162889}\leq 4.81\cdot 10^9$. 
\end{proof}

We will also make use of the following estimate for $\theta(x)$ from Broadbent et al. \cite[Theorem 1]{broadbent2020}. 
\begin{thm}[Broadbent et al.]\label{thm:lethbridge_thm}
For $x > e^{20} \approx 3.59\cdot10^9$, we have
 \begin{equation}\label{eqn:lethbridge_k}
      \lvert \theta(x) - x \rvert \leq 0. 375 \frac{x}{\log^3 x}.
 \end{equation}
\end{thm}

\section{An Important Lemma}\label{sec:analytic_bounds}

Theorem \ref{thm:dudek} is true as long as $R(n) > 0$ for all $n \geq 3$.
To prove Theorem \ref{thm:dudek}, Dudek found an explicit lower bound for $R(n)$ in \cite[Section 2.3]{dudek2017}, which inferred that $R(n) > 0$ for all $n \geq 10^{10}$, then computationally checked $R(n) > 0$ for $n \in [3, 10^{10})$.

We will prove Lemma \ref{thm:analytic_bound} and use it to prove the main results of this paper. The main benefit of this (over the lower bound in \cite[Section 2.3]{dudek2017}) is that we will need to manually verify our main results for a smaller range of $n$.

\begin{lem}\label{thm:analytic_bound}
Suppose $A\in (0, 1/2)$ and $n \geq 4.81\cdot 10^9$, then
\begin{align}
    \frac{R(n)}{n}
    > 0.37395 - \frac{0.95}{\log n} - \frac{0.375}{\log^3 n} &- 0.0096 \left(\frac{1+2A}{1-2A}\right)\nonumber\\
    &- \log n\left(n^{-2A} + n^{-A} - n^{A - 1} + n^{-\frac{1}{2}}\right), \label{eqn:analytic_bound_result}
\end{align}
where $0.37395$ is Artin's constant, rounded to 5 decimal places.
\end{lem}

The improvements we obtain come from Proposition \ref{prop:bennett_prop} (which is wider-reaching than the results from Ramar\'{e}--Rumely \cite{ramare1996} which Dudek used), and Theorem \ref{thm:lethbridge_thm}.

\subsection{Set-up} 

Trivially, if $(a,n) > 1$, then $\theta(n;a^2,n) \leq \log n$. Therefore,
$$R(n) > \sum_{\substack{a\leq n^{\frac{1}{2}}\\(a,n)=1}}\mu(a)\theta(n;a^2,n) - n^{\frac{1}{2}}\log n = \Sigma_1 + \Sigma_2 + \Sigma_3 - n^\frac{1}{2}\log{n},$$
in which  $A \in \left(0,1/2\right)$ will be chosen later,
\begin{align*}
    \Sigma_1 = \sum_{\substack{a \leq 316 \\ (a,n) = 1}} \mu(a)\theta(n;a^2,n),\quad
    \Sigma_2 &= \sum_{\substack{316 < a \leq n^A \\ (a, n) = 1}} \mu(a)\theta(n;a^2,n),\mbox{ and}\\
    \Sigma_3 &= \sum_{\substack{n^A < a \leq n^\frac{1}{2} \\ (a, n) = 1}} \mu(a)\theta(n;a^2,n).
\end{align*}
We will bound $\Sigma_1 + \Sigma_2$ and $\Sigma_3$ separately.

\subsection{Bounding $\Sigma_1 + \Sigma_2$}\label{subsec:first}

We start by listing important bounds, which we will use to deduce a lower bound for $\Sigma_1 + \Sigma_2$.
First, we observe by computation that
\begin{equation}\label{eq:uno}
    \sum_{2\leq a\leq 316} c_{\theta}(a^2) = 0.9474935 < 0.95.
\end{equation}
Second, suppose that $c$ denotes Artin's constant. It follows from computations by Wrench \cite{wrench1961} that
\begin{equation}\label{eq:dos}
    \sum_{(a,n) = 1} \frac{\mu(a)}{\varphi(a^2)} > \prod_p \left(1 - \frac{1}{p(p - 1)}\right) = c > 0.37395.
\end{equation}
Third, in the range $316 < a \leq n^A$, the Brun--Titchmarsh theorem \cite{montgomery1973} yields
\begin{equation}\label{eq:tres_Brun_Titchmarsh}
    \theta(n;a^2,n) = \frac{n}{\varphi(a^2)} + \varepsilon \left(\frac{1+2A}{1-2A}\right)\frac{n}{\varphi(a^2)},
\end{equation}
such that $|\varepsilon | < 1$.
Finally, we may observe that
\begin{equation}\label{eq:quattro}
    \sum_{\substack{a>316\\(a,n)=1}}\frac{\mu(a)}{\varphi(a^2)}
    \leq \sum_{a=1}^\infty \frac{\mu^2(a)}{\varphi(a^2)} - \sum_{a\leq 316}\frac{\mu^2(a)}{\varphi(a^2)}
    < 0.0096.
\end{equation}
To deduce the upper bound, we observed that an upper bound for the infinite sum is $1.95$ \cite{ramare1995} and computed the finite sum manually. Computations suggest that \eqref{eq:quattro} may be numerically improved, but this is unnecessary for our purposes.


To bound $\Sigma_1$, combine Proposition \ref{prop:bennett_prop}, Theorem \ref{thm:lethbridge_thm}, \eqref{eq:uno} and \eqref{eq:dos} to yield
\begin{align*}
    \Sigma_1
    &> n \left( \sum_{\substack{2\leq a \leq 316 \\ (a,n) = 1}} \frac{\mu(a)}{\varphi(a^2)} - \sum_{\substack{2\leq a \leq 316 \\ (a,n) = 1}} \frac{c_{\theta}(a^2)\mu(a)}{\log n}  + 1 - \frac{0.375}{\log^3 n}\right)\\
    &>  n \left( \sum_{(a,n)=1} \frac{\mu(a)}{\varphi(a^2)} - \sum_{\substack{a > 316 \\ (a,n) = 1}} \frac{\mu(a)}{\varphi(a^2)} - \sum_{\substack{2\leq a\leq 316\\(a,n)=1}}\frac{c_{\theta}(a^2)}{\log n} - \frac{0.375}{\log^3 n}\right)\\
    &>  n \left( 0.37395 - \sum_{\substack{a > 316 \\ (a,n) = 1}} \frac{\mu(a)}{\varphi(a^2)} - \frac{0.95}{\log n} - \frac{0.375}{\log^3 n}\right),
\end{align*}
since the $+1$ gets absorbed into the left-most sum. Next, use \eqref{eq:tres_Brun_Titchmarsh} to see that
\begin{align*}
    \Sigma_2
    &> n\left(\sum_{\substack{316 < a \leq n^A \\ (a, n) = 1}} \frac{\mu(a)}{\varphi(a^2)} - \left(\frac{1 + 2A}{1 - 2A}\right) \sum_{\substack{316 < a \leq n^A \\ (a, n) = 1}} \frac{\mu^2(a)}{\varphi(a^2)} \right).
\end{align*}
Finally, using the preceding observations and \eqref{eq:quattro}, $\Sigma_1 + \Sigma_2$ is larger than
\begin{align*}
    n &\left(0.37395 - \sum_{\substack{a > n^A \\ (a, n) = 1}} \frac{\mu(a)}{\varphi(a^2)} - \frac{0.95}{\log n} - \frac{0.375}{\log^3 n}-\left(\frac{1+2A}{1-2A}\right)\sum_{\substack{316<a\leq n^A\\(a,n)=1}}\frac{\mu^2(a)}{\varphi(a^2)} \right)\\
    &\quad > n \left(0.37395 - \frac{0.95}{\log n} - \frac{0.375}{\log^3 n} -\left(\frac{1+2A}{1-2A}\right)\sum_{\substack{a>316\\(a,n)=1}}\frac{\mu^2(a)}{\varphi(a^2)} \right)\\
    &\quad > n \left(0.37395 - \frac{0.95}{\log n} - \frac{0.375}{\log^3 n} - 0.0096 \left(\frac{1+2A}{1-2A}\right) \right).
\end{align*}

\subsection{Final steps}

Following a similar logic to Dudek \cite{dudek2017}, with less waste, we used a trivial bound for $\theta(n;a^2,n)$ to bound $|\Sigma_3|$ and obtain
\begin{equation}\label{eqn:Sigma_3}
    \Sigma_3 > - n \log n\left(n^{-2A} + n^{-A} - n^{A - 1}\right).
\end{equation}
Combining our preceding observations, we have established \eqref{eqn:analytic_bound_result} for all $n\geq 4.81\cdot 10^9$.

%

\section{Main Results}\label{sec:main_results}

In this section, we will establish all of the main results of this paper.
We treat Theorem \ref{thm:one_prime_two} separately, but to prove the remaining results, we will use Lemma \ref{thm:analytic_bound} to establish them for large $n$ and the algorithm described in Section \ref{sec:computation} to verify the results for small $n$.

Suppose that $2\leq q\leq 10^5$ is prime and $R_q(n)$ denotes the weighted number of representations of $n$ as the sum of a prime and a square-free number coprime to $q$. Then, we have
$$R_q(n) = \sum_{\substack{p \leq n \\ p \not\equiv n \imod{q}}} \mu^2(n-p)\log{p} = R(n) - \sum_{\substack{p \leq n \\ p \equiv n \imod{q}}} \mu^2(n-p)\log{p}.$$
Therefore, to show $R_q(n) > 0$, it suffices to demonstrate
\begin{equation}\label{eqn:rvsrq}
    R(n) > \sum_{\substack{p \leq n \\ p \equiv n \imod{q}}} \mu^2(n-p)\log{p}.
\end{equation}

\subsection{Proof of Theorem \ref{thm:one_prime_two}}\label{ssec:two}

It is an equivalent problem to establish $R_2(n)>0$ for all even $n > 2$. First, we observe that for each \textit{odd} $n$, $n-p$ is even for all odd primes $p$, hence
\begin{equation*}
    R_2(n) = \begin{cases} \log(n-2) & \mbox{if }\mu^2(n - 2) = 1,\\0 & \mbox{if }\mu(n - 2) = 0. \end{cases}
\end{equation*}
There are infinitely many odd choices for $n$ such that $\mu(n - 2) = 0$, thence our restriction to \textit{even} $n$ in this case.

Suppose that $n$ is \textit{even}, then $R_2(n) > 0$ if and only if $R(n) > \theta(n;2,n)$.
If $(n, q) > 1$, then $\theta(n;q,n) \leq \log{q}$. Therefore it suffices to show that $R(n) > \log 2$.
If $n \geq 4$, then Theorem \ref{thm:dudek} guarantees that there exists at least one prime $p \in (2,n)$ such that $\mu^2(n-p)=1$. Note that $p\neq 2$, because $n-2 \geq 2$ is even for even $n$, so $(n-2,2)\neq 1$.
It follows that there exists a prime $p \in (2,n)$ such that $R(n) > \log{p} > \log{2}$. 
\qed

\subsection{Proof of Theorem \ref{thm:one_prime_general} for large $n$}\label{ssec:one_gen}

Suppose $3 < q \leq 10^5$ is prime, then Theorem \ref{thm:one_prime_general} holds for $n \geq 8\cdot 10^9$ if and only if $R_q(n) > 0$. Using \eqref{eqn:rvsrq}, it suffices to show that $R(n) > \theta(n;q,n)$.
By Lemma \ref{thm:analytic_bound}, we need to show that there exists $A\in (0,1/2)$ such that
\begin{align}
    0.37395 - \frac{0.95}{\log{n}} &- \frac{0.375}{\log^3{n}} - 0.0096 \left(\frac{1+2A}{1-2A}\right)\nonumber\\
    &- \left(n^{-2A} + n^{- A} - n^{A-1} + n^{-\frac{1}{2}}\right)\log{n} > \frac{\theta(n;q,n)}{n}\label{eqn:sufficientcheck}.
\end{align}
We may use Proposition \ref{prop:bennett_prop} to estimate $\theta(n;q,n)$ for each $3 \leq q \leq 10^5$. It follows that $A = 0.33$ implies \eqref{eqn:sufficientcheck} for all primes $q$ in our assumed range. 
\qed

\subsection{Proof of Theorem \ref{thm:one_prime_three} for large $n$}\label{ssec:one_three}

We use a similar method to the preceding proof, although we must consider the case $q=3$ separately because it is clear that $1/\varphi(3) = 1/2 > 0.37395$. So, we will need to consider a stronger version of \eqref{eqn:sufficientcheck}.
Observe that 
\begin{equation}\label{eqn:mod3breakdown}
   \sum_{\substack{p \leq n \\ p \equiv n \imod{3}}}\mu^2(n-p)\log{p} = \sum_{\substack{p \leq n \\ p \equiv n \imod{3}}} \log{p} - \sum_{\substack{p \leq n \\ p \equiv n \imod{3} \\ \mu(n-p) = 0}}\log{p}.
\end{equation}
An inclusion-exclusion argument yields
\begin{align*}
    \sum_{\substack{p \leq n \\ p \equiv n \imod{3} \\ \mu(n-p) = 0}}\log{p}
    &> \sum_{\substack{p \leq n \\ p \equiv n \imod{9} \textrm{ or } \\  p \equiv n \imod{12} \textrm{ or } \\ p \equiv n \imod{75} } } \log{p} \\
    &= \theta(n;9,n) + \theta(n;12,n) + \theta(n;75,n) \\ 
    &\qquad - \theta(n;36,n) - \theta(n;225,n) - \theta(n;300,n) + \theta(n;900,n).
\end{align*}
Therefore, \eqref{eqn:mod3breakdown} yields
\begin{align}
    \sum_{\substack{p \leq n \\ p \equiv n \imod{3}}}&\mu^2(n-p)\log{p} < \theta(n;3,n) - \theta(n;9,n) - \theta(n;12,n) - \theta(n;75,n)\\ 
    &\qquad + \theta(n;36,n)+ \theta(n;225,n) + \theta(n;300,n) - \theta(n;900,n).\label{eqn:allthetas}
\end{align}
Using the explicit bounds from Bennett et al. \cite[Theorem 1.2]{bennett2018} to estimate each $\theta(n;q,n)$ term in \eqref{eqn:allthetas} according to these values establishes
\begin{equation}\label{eqn:modupperbound}
    \sum_{\substack{p \leq n \\ p \equiv n \imod{3}}}\mu^2(n-p)\log{p} < \frac{19}{120}n + 0.00592\frac{n}{\log{n}}.
\end{equation} 

We may compare \eqref{eqn:modupperbound} with Lemma \ref{thm:analytic_bound}, and thereby establish \eqref{eqn:rvsrq} whenever 
\begin{align}
    0.37395 - \frac{0.95}{\log n} &- \frac{0.375}{\log^3{n}} - 0.0096 \left(\frac{1+2A}{1-2A}\right) \label{eqn:finalcheck}\\
    &- \left( n^{-2A} + n^{- A} + n^{A-1} - n^{-\frac{1}{2}}\right)\log n  > \frac{19}{120} + \frac{0.00592}{\log{n}}.\nonumber
\end{align}
Choosing $A = 0.33$ will verify that \eqref{eqn:finalcheck} holds for $n \geq 8\cdot 10^9$.
\qed

\subsection{Proof of Corollary \ref{thm:two_prime_general} for large $n$}\label{subsec:extension}


Suppose that $2\leq q\leq 10^5$ is prime. We will consider the cases $q=2$ and $q \geq 3$ separately. In the former case, we did not require the computations (which will be outlined in section \ref{sec:computation}) to verify that the result is true for small $n$, so we consider a larger range for $n$ in this case.

If $q = 2$ and $n > 4$ is \textit{even}, then $n - 2 > 2$ is also even. Therefore, there exists a prime $p_1$ and odd square-free number $\eta_1$ such that $n - 2 = p_1 + \eta_1$ by Theorem \ref{thm:one_prime_two}. Moreover, if $q = 2$ and $n > 5$ is \textit{odd}, then $n - 3 > 2$ is even. Therefore, there exists a prime $p_2$ and odd square-free number $\eta_2$ such that $n - 3 = p_2 + \eta_2$ by Theorem \ref{thm:one_prime_two}. To complete Corollary \ref{thm:two_prime_general} at $q = 2$, observe that $5 = 2 + 2 + 1$.

If $q \geq 3$ and $n \geq 8\cdot 10^9$, then suppose that
\begin{equation*}
    T(n)
    := \sum_{\substack{p\leq n\\n - p \not\in \{1,2,11\}}}\mu^2(n-p)\log p 
    > R(n) - 3 \log n .
\end{equation*}
If $T(n) > 0$ then there exists at least one prime $p_3$ such that $n-p_3 > 2$ and $n-p_3\neq 11$. Hence, Corollary \ref{thm:two_prime_general} is also true by corollary of Theorem \ref{thm:one_prime_general} for $q > 3$, and by corollary of Theorem \ref{thm:one_prime_three} for $q=3$.
Therefore, it suffices to show $T(n) > 0$ in the desired range of $n$.
It follows from Lemma \ref{thm:analytic_bound} that
\begin{align}
    \frac{T(n)}{n}
    > 0.37395 - \frac{0.95}{\log n} &- \frac{0.375}{\log^3 n} - 0.0096 \left(\frac{1+2A}{1-2A}\right)\nonumber\\
    &- \log n\left(n^{-2A} + n^{-A} - n^{A - 1} + n^{-\frac{1}{2}} + \frac{3}{n}\right).\label{eqn:final_appl_3_large_n}
\end{align}
Now, \eqref{eqn:final_appl_3_large_n} with $A=0.385$ implies the result for large $n$.
\qed

\subsection{Computations}\label{sec:computation}

To complete each of our main results, we verified each result for small $n$, wherever necessary. We did this computationally, by slightly adapting the algorithm used by Dudek in \cite[p.~239]{dudek2017}. Our computations took just short of 7 hours on a machine equipped with 3.20 GHz CPU, using Maple${}^\mathrm{TM}$ \footnote{Maple is a trademark of Waterloo Maple, inc.}.

If $3 < n \leq 4\cdot10^{18}$ is even, we know by Oliveira e Silva et al. \cite{oliveira2014} that $n$ is the sum of two primes. Unless $n = q+q$ for some prime $q \in \left[3, 10^5\right]$, we are done. When $n = q + q$, it is a simple task to verify that it has at least one other representation as a prime plus a square-free co-prime to $q$. Hence, we only need to consider odd integers between $3$ and $8\cdot10^{8}$. 

As in Dudek's algorithm, we pre-compute a set $S$ of square-free numbers up to $2\cdot10^7$. We break the problem up, considering $n$ in intervals of the form
\[I_a = \left[a\cdot10^7, (a+1)\cdot10^7\right),\] 
where $a$ is an integer between 1 and 800. For each such $a$, we compute decreasing lists $P_a = \left(p_1, p_2,\ldots, p_{100}\right)$ of the 100 largest primes in $I_{a-1}$. Starting with the smallest odd $n$ in $I_a$, we check if $n-p_i$ is in $S$ as $i$ ranges from 1 to 100. Each time this check is successful, we compute the $\gcd$ of $n-p_i$ with all previous successful $n-p_j$, moving on to $n+2$ when this $\gcd$ equals 2 (that is, when there is a representation with a square-free number co-prime to every prime $q  \in \left[3, 10^5\right]$). If there were any $n$ for which the largest 100 primes did not produce all the appropriate representations, we could have checked these cases separately with more primes. However, our program did not return any such $n$.

For the initial interval $n \in \left(2, 10^7\right)$, a similar check can be used. Relevant representations can easily be found for $n$ up to $10^6$, with the exception of $n = 2$ and $n = 11$ (which is an exception only when $q = 3$). Then, letting $P_0$ be the set of the 100 largest primes less than $10^6$, we perform the same check as we did for the other intervals to $n \in (10^6, 10^7)$, finding no new exceptions.

To verify Corollary \ref{thm:two_prime_general}, we adapted the algorithm above. Note that we only need to check the even $n$ in this scenario, since the result follows directly from the ternary Goldbach conjecture for odd $n$ \cite{helfgott2013}. Importantly, we used $S' = S \setminus \{1,2,11\}$ in place of $S$ and no exceptions were found for $n$ between 5 and $8\cdot10^9$.

\bibliographystyle{amsplain}
\bibliography{references}

\providecommand{\bysame}{\leavevmode\hbox to3em{\hrulefill}\thinspace}
\providecommand{\MR}{\relax\ifhmode\unskip\space\fi MR }
\providecommand{\MRhref}[2]{%
  \href{http://www.ams.org/mathscinet-getitem?mr=#1}{#2}
}
\providecommand{\href}[2]{#2}
\begin{thebibliography}{10}

\bibitem{bennett2018}
M.~A. Bennett, G.~Martin, K.~O'Bryant, and A.~Rechnitzer, \emph{Explicit bounds
  for primes in arithmetic progressions}, Illinois J. Math. \textbf{62} (2018),
  no.~1-4, 427--532. \MR{3922423}

\bibitem{broadbent2020}
S.~Broadbent, H.~Kadiri, A.~Lumley, N.~Ng, and K.~Wilk, \emph{Sharper bounds
  for the {C}hebyshev function $\theta(x)$}, arXiv preprint arXiv:2002.11068
  (2020).

\bibitem{Chen66}
J.~R. Chen, \emph{On the representation of a large even integer as the sum of a
  prime and the product of at most two primes}, Kexue Tongbao \textbf{17}
  (1966), 385--386. \MR{207668}

\bibitem{Chen78}
\bysame, \emph{On the representation of a large even integer as the sum of a
  prime and the product of at most two primes. {II}}, Sci. Sinica \textbf{21}
  (1978), no.~4, 421--430. \MR{511293}

\bibitem{dudek2017}
A.~W. Dudek, \emph{On the sum of a prime and a square-free number}, Ramanujan
  J. \textbf{42} (2017), no.~1, 233--240.

\bibitem{estermann1931}
T.~Estermann, \emph{On the representations of a number as the sum of a prime
  and a quadratfrei number}, J. Lond. Math. Soc. \textbf{1} (1931), no.~3,
  219--221.

\bibitem{helfgott2013}
H.~A. Helfgott, \emph{The ternary {G}oldbach conjecture is true}, arXiv
  preprint arXiv:1312.7748 (2013).

\bibitem{montgomery1973}
H.~L. Montgomery and R.~C. Vaughan, \emph{The large sieve}, Mathematika
  \textbf{20} (1973), no.~2, 119--134.

\bibitem{oliveira2014}
T.~Oliveira~e Silva, S.~Herzog, and S.~Pardi, \emph{Empirical verification of
  the even {G}oldbach conjecture and computation of prime gaps up to $4\cdot
  10^{18}$}, Math. Comput. \textbf{83} (2014), no.~288, 2033--2060.

\bibitem{ramare1995}
O.~Ramar{\'e}, \emph{On \v{S}nirel'man's constant}, Ann. Sc. Norm. Super. Pisa
  Cl. Sci. \textbf{22} (1995), no.~4, 645--706.

\bibitem{ramare1996}
O.~Ramar{\'e} and R.~Rumely, \emph{Primes in arithmetic progressions}, Math.
  Comput. \textbf{65} (1996), 397--425.

\bibitem{wrench1961}
J.~W. Wrench~Jr, \emph{Evaluation of {A}rtin's constant and the twin-prime
  constant}, Math. Comput. (1961), 396--398.

\bibitem{Yamada2015}
T.~Yamada, \emph{Explicit {C}hen’s theorem}, arXiv preprint arXiv:1511.03409
  (2015).

\bibitem{yau2019}
K.~H. Yau, \emph{Representation of an integer as the sum of a prime in
  arithmetic progression and a square-free integer}, Funct. et Approx. Comment.
  Math. (to appear).

\end{thebibliography}

\end{document}